\documentclass[a4paper]{article}

\usepackage[english]{babel}
\usepackage[utf8]{inputenc}
\usepackage{amsmath,amsthm,amsfonts,graphics}
\usepackage{graphicx}
\usepackage{authblk}

\theoremstyle{plain}
\newtheorem{theorem}{Theorem}
\newtheorem{lemma}[theorem]{Lemma}
\newtheorem{corollary}[theorem]{Corollary}
\theoremstyle{definition}
\newtheorem{definition}[theorem]{Definition}

\newcommand{\mat}[1]{\mathsf{#1}}
\renewcommand{\vec}[1]{\mathbf{#1}}
\newcommand{\one}{\mathtt{1}}
\newcommand{\zero}{\mathtt{0}}
\newcommand{\zo}{\{\mathtt{0},\mathtt{1}\}}

\renewcommand{\Pr}{\mathbb{P}}
\newcommand{\ee}{\mathrm{e}}
\newcommand{\K}{\mathcal{K}}
\renewcommand{\L}{\mathcal{L}}
\newcommand{\mbyn}{m\times n}

\title{Almost Separable Matrices}

\author[1]{Matthew Aldridge\thanks{m.aldridge@bristol.ac.uk}}
\author[2]{Leonardo Baldassini\thanks{leonardo.baldassini@bristol.ac.uk}}
\author[1]{Karen Gunderson\thanks{k.gunderson@bristol.ac.uk}}
\affil[1]{Heilbronn Institute for Mathematics Research, School of Mathematics, University of Bristol, Bristol, UK}
\affil[2]{School of Mathematics, University of Bristol, Bristol, UK}

\date{\today}

\begin{document}

\maketitle

\begin{abstract}
An $\mbyn$ matrix $\mat A$ with column supports $\{S_i\}$
is $k$-separable if the disjunctions $\bigcup_{i \in \K} S_i$
are all distinct over all sets $\K$ of cardinality $k$.
While a simple counting bound
shows that $m > k \log_2 n/k$ rows are required for a separable matrix to exist, in fact
it is necessary for $m$ to be about a factor of $k$ more than this.
In this paper, we consider a weaker definition of `almost $k$-separability', which requires
that the disjunctions are `mostly distinct'. We show using a random construction that these matrices
exist with $m = O(k \log n)$ rows, which is optimal for $k = O(n^{1-\beta})$. Further, by calculating
explicit constants, we show how almost separable
matrices give new bounds on the rate of nonadaptive group testing.
\end{abstract}

\section{Introduction}

Let $\mat A \in \zo^{\mbyn}$ be an $\mbyn$ binary matrix,
and write $S_i$ for the support of its $i$th column (that is, the locations
of the $\one$s).
Then $\mat A$ is said to be \emph{$k$-separable} if the sets $\bigcup_{i \in \K} S_i$
are all distinct over all sets $\K\in\{1, 2, \ldots, n\}$ of cardinality $k$ (see
Definition \ref{defsep}, to come).

Separable matrices were first introduced by Erd\H{o}s and Moser in 1970 \cite{erdosmoser}
and have since been studied in different contexts, including coding theory, 
combinatorics and, as we discuss later, group testing, where they play a very important role.

Separable matrices are often studied through the slightly stronger concept of {\em disjunct
matrices} (see Definition \ref{disj}). 
Disjunct matrices were first introduced by Kautz and Singleton
\cite{ks} and, just like separable matrices, they have been extensively studied in coding theory,
combinatorics and group testing \cite{chen,hwang,dyachkov,furedi,ruszinko}. 

A central question in the study of both separable and disjunct matrices is the following: 
Given $n$ and $k$, how large must $m$ be
for there to exist either an $\mbyn$ $k$-separable or disjunct matrix? In this paper, we investigate
the asymptotics for separability as $n \to \infty$, where $k$ may grow with $n$.

A simple counting bound (Theorem \ref{counting}) shows that 
$m\geq \Omega(k \log n/k)$ rows are required.
Disappointingly, when $k = o(n)$ this bound is not tight, and we require roughly a factor of $k$
more than this, as in fact it has been shown \cite{hwang,chen} that $m \geq \Omega(k^2 \log n / \log k)$ is needed.
This lower bound is motivated by the connection between disjunctness and separability, as we 
discuss in Section \ref{sepsec}.
Notice that when $k$ grows linearly with $n$, taking the identity matrix is order optimal 
-- for this reason, we consider only $k = o(n)$ in this paper.

In order to meet the lower bound $m\geq\Omega(k\log n/k)$, 
we consider a relaxation of the requirement of $k$-separability
to \emph{almost $k$-separability}. Roughly speaking, a matrix
is \emph{almost $k$-separable} if the sets $\bigcup_{i \in \K} S_i$ are
`usually' distinct -- see Definition \ref{defalmsep} for a formal definition. 

Our main result shows that it is possible to 
achieve almost separability with only $O(k \log n)$ rows (Theorem \ref{mainthm}).
When $k = O(n^{1-\beta})$, for any $\beta \in (0,1]$, this is order-optimal to
the counting bound. However, we also aim to get best possible constants for $m$ -
a goal motivated by the study of the rate of group testing algorithms. 

Group testing is an old and well-studied 
search problem, first considered by Dorfman \cite{dorfman}, where the goal is to recover a sparse subset of $k$ {\em defective} elements 
spread among $n$ otherwise identical items. Instead of testing each item for defectiveness individually, classic 
group testing algorithms test items in batches. In the noiseless binary model we consider, 
tests can only reveal whether a given set contains at least one defective (a positive test) or no defectives (a negative test).
The connection between separable matrices and nonadaptive group testing is
well-known, and we discuss it in Section \ref{gtsec}. For the moment, we just observe 
that a sequence of tests designed a priori ({\em nonadaptive} group testing) has a natural binary-matrix representation:
each length-$n$ row represents a test, with entries being $\one$ if 
the corresponding item is being included in the test.

A matrix being
$k$-separable is equivalent to having zero probability of error for nonadaptive
group testing, while a matrix being almost
$k$-separable is equivalent to having a small probability of error. 
The `arbitrarily small probability of error' criterion we consider
here is the same as that in Shannon's theory of channel coding.

With this comparison in mind, we consider
the concept of \emph{rate} of group testing (Definition \ref{defrate}) for $k = n^{1 - \beta}$
defective items in a population of size $n$, which can be thought of as the
amount of information conveyed by each test. Using a separable matrix
with $m = \Omega(k^2 \log n / \log k)$ rows leads to a group testing rate of $0$.
However, using an almost separable matrix with $m =O(k \log n)$ rows
gives a strictly positive rate, with the rate depending on the contant implied by the big-$O$. Hence, 
here we are interested in getting good constants for $m$, not only in order-wise results.

In Theorem \ref{ratethm}, we show that our results meet previous results for the limiting regime where $k$ is fixed
as $n \to \infty$, and improves over the previous
best known bounds for larger values of the {\em sparsity parameter} $\beta\in[0,1]$ 
in the $k=n^{1-\beta}$ regime frequently considered in the group testing literature.

\section{Separable matrices}
\label{sepsec}

We begin by recalling the definition of a separable matrix.

\begin{definition} \label{defsep}
Given an $\mbyn$ binary matrix $\mat A = (a_{ij}) \in \zo^{\mbyn}$, we shall write
  $S_i := \{ j : a_{ij} = \one\}$
for the \emph{support} of column $i$ and for $\K \subseteq \{1,2,\dots, n\}$ also write
$S(\K) := \bigcup_{i \in \K} S_i$ for the support of a disjunction of columns.

  The matrix $\mat A$ is called \emph{$k$-separable} matrix
  if the for all sets $\K$ of size $k$, there is no other set $\L$ also of size $k$ with
  $S(\L)$ = $S(\K)$.
%
\end{definition}

The case $k = 0$ is trivial, so we assume $k \geq 1$ throughout. We shall also assume $k \leq n/2$, which will be no restriction in the limiting regimes we study.


The following counting bound is described by Chen and Hwang as ``simple-minded'' \cite{chen}.

\begin{theorem} \label{counting}
  Let $M(n,k)$ be the smallest $m$ such that an $\mbyn$ $k$-separable
  matrix exists. Then
  \[ M(n,k) \geq \log_2 \binom nk . \]
\end{theorem}

\begin{proof}
  Clearly 
    \[ \left| \left\{ S(\K) :|\K| = k \right \} \right| \leq |\mathcal P (\{1,2,\dots,m\})| = 2^m ,\]
  where $\mathcal{P}$ denotes the power set.
  Hence for $\mat A$ to be $k$-separable we require $2^m \geq \binom nk$,
  and taking logarithms gives the result. \qed
\end{proof}

Using the lower bound of
  \begin{equation} \label{binom}
    \left(\frac{n}{k}\right)^k \leq \binom nk \leq \left(\frac{\ee n}{k}\right)^k  
  \end{equation}
(which we shall use many times in this paper),
we see that a $k$-separable matrix must have at least $m \geq k \log_2 n/k = \Omega(k \log n/k)$ rows.

As we anticipated, separable matrices are tightly related to another class of matrices, namely 
that of disjunct matrices. 

\begin{definition} \label{disj}
With the notation of Definition \ref{defsep}, $\mat A$ is \emph{$k$-disjunct} if for all
sets $\K$ of cardinality $|\K| = k$, there does not exist $i \not\in \K$ such that
$S_i \subseteq S(\K)$.
\end{definition}

In the language of set systems, a matrix $\mat A$ being $k$-seperable  is equivalent to the family $\{S_i\}_{i = 1}^n$ being $k$-union-free, and $\mat A$ being $k$-disjunct is equivalent to $\{S_i\}_{i = 1}^n$ being $k$-cover-free.

It's easy to see that $k$-disjunctness implies $k$-separability (see, for example, \cite{ks}, \cite[Section 7.2]{hwang}, or the special case $\epsilon = 0$ of Lemma \ref{implies} below). On the other hand, Chen and Hwang \cite[Theorem 2]{chen} 
have shown that it is possible to construct 
a $k$-disjunct matrix from a $2k$-separable matrix by adding at most one row to it, which means that 
disjunct and separable matrices share the same order-wise asymptotics. Dyachkov and Rykov have quantified these asymptotics by showing that $m \geq \Omega(k^2 \log n / \log k)$ rows are necessary for a matrix to be $k$-disjunct \cite{dyachkov} -- similar results appear elsewhere \cite{ruszinko} \cite{furedi} \cite[Theorem 7.2.14]{hwang}. This means that it is not possible to create a $k$-separable matrix with 
$m = O(k \log n)$ rows.

As disjunctness is a stronger (and, in some ways, simpler) property than separability, 
efforts to derive upper bounds on $m$ for separable matrices have often proceeded via the construction of disjunct matrices.
In their seminal paper \cite{ks}, Kautz and Singleton give a probabilistic existence theorem for
$k$-disjunct matrices with $m = O(k^2 \log n)$ rows. In the group testing literature there 
exist explicit constructions of testing schemes with $O(k^2 \log n)$ rows, 
see for example Porat and Rothschild \cite{porat}.


\section{Almost separable matrices}

Since separable matrices cannot meet the counting bound, it would be of interest if a matrix could be 
close to being separable using only $O(k \log n)$ rows.  Such a matrix would be order-optimal.

With this in mind, we define the concept
of an \emph{almost separable} matrix in a similar manner to Defintion \ref{defsep}.

\begin{definition} \label{defalmsep}
  With the notation of Definition \ref{defsep}, $\mat A \in \zo^{\mbyn}$ is
  \emph{$\epsilon$-almost $k$-separable}
  if for at most $\epsilon \binom nk$ sets $\K$ of size $k$ does there exist another
  set $\L$ of size $k$ with $S(\L) = S(\K)$.
%
\end{definition}

An analogous definition is present 
in
for example \cite{zhigljavsky}, where almost separable matrices are called {\em weakly separating 
designs}. Note that setting $\epsilon = 0$ gives the definition of a separable matrix.


The main result of this paper is to show the existence of $\epsilon$-almost
$k$-separable matrices with $m = O(k \log n)$ rows (see Theorem \ref{mainthm}
below). We also examine the implicit constants for the case when $k = n^{1 - \beta}$
grows polynomially in $n$.

Malyutov \cite{malyutov} effectively showed that $\epsilon$-almost $k$-separable matrices exist with $m = (k + o(1)) \log_2 n$ rows in the regime where $k$ is fixed as $n \to \infty$.
This is a special case of a more general result Malyutov proved using an information theoretic argument -- this 
and similar work is reviewed in \cite{malyutovsurvey}. 
Seb\H o showed effectively the same result \cite{sebo}, again for fixed $k$, by analysing a concrete bound on the probability that
there are two different sets of size $k$ whose disjunctions coincide -- we follow a similar route here later.
The same result for $k$ fixed and $n\rightarrow\infty$ was rediscovered by Zhigljavsky 
\cite[Theorem 5.5]{zhigljavsky}
. Although technically different from Seb\H o's argument, 
Zhigljavsky's proof is morally similar: given two sets $\K$ and $\L$ of $k$ columns each, 
Zhigljavsky counts how many rows it is possible to construct that would produce the 
same value for both $S(\K)$ and $S(\L)$. He calls this number a R\' enyi coefficient and 
only considers designs with fixed- or bounded-size tests.

Our result improves on these by allowing $k$ to vary arbitrarily with $n$, subject to $k = o(n)$.
In our discussion of group testing in Section \ref{gtsec} we show how, in some regimes, this work 
also improves on recent results on nonadaptive group testing giving bounds of the form $m = O(k \log n)$.

The definition of a disjunct matrix (Defintion \ref{disj}) can similarly be weakened to give an \emph{almost disjunct matrix}. (This definition also appears in \cite{mazumdar} and, previously, in \cite{macula}.)

\begin{definition}
With the notation of Definition \ref{defsep}, $\mat A$ is \emph{$\epsilon$-almost $k$-disjunct} if for at most $\epsilon \binom nk$ sets $\K$ of size $k$ does there exist a column $i \not\in \K$
  with $S_i \subseteq S(\K)$.
\end{definition}

Note again that $\epsilon = 0$ corresponds to a disjunct matrix. Unsurprisingly, almost disjunctness implies almost separability.

\begin{lemma} \label{implies}
Let $\mat A$ be an $\epsilon$-almost $k$-disjunct matrix. Then $\mat A$ is $\epsilon$-almost $k$-separable (with the same $\epsilon$ and $k$).
\end{lemma}

\begin{proof}
We prove the contrapositive. Suppose $\mat A$ is not $\epsilon$-almost $k$-separable. Then there are more than $\epsilon \binom nk$ sets of size $k$ breaking separability.
Let $\K$ be one of these sets, so there is another set $\L$ of size $k$ with $S(\K) = S(\L)$. 
Letting $i \in \K \setminus \L$, we have $S_i \subseteq S(\K)$, breaking disjunctness. 
Hence there are more than $\epsilon \binom nk$ sets breaking disjunctness, and
$\mat A$ is not $\epsilon$-almost $k$-disjunct. \qed
\end{proof}

Mazumdar \cite{mazumdar} shows that there exist almost $k$-disjunct matrices with $m=O(k^{3/2} \sqrt{\log n})$ rows
in the regime $k\sim n^\delta$, $\delta >0$, which is the same as that we consider for group testing. 
Mazumdar's construction is similar to those of Kautz and Singleton \cite{ks} and Porat and 
Rothschild \cite{porat}. In particular, \cite{ks} shows how to build fully disjunct matrices with $O(k^2 \log^2_{k\log n}n)$
rows by mapping the symbols of a $q$-ary Reed-Solomon code to unit-weight binary vectors of length $q$, while \cite{porat} 
improves on this scheme by replacing the RS code with a linear $q$-ary code achieving the Gilbert-Varshamov bound. This 
produces fully disjunct matrices with $O(k^2\log n)$ rows. 
This improves on the $\Omega(k^2 \log n / \log k)$ required for full disjunctness or separability, 
while being less good than the $O(k \log n)$ we achieve for almost separability here.

Our main result is then the following.

\begin{theorem} \label{mainthm}
For any sequence $k = k(n) = o(n)$ and $\epsilon > 0$, there exist an $\epsilon$-almost $k$-separable matrix
with $m = O(k \log n)$ rows.

More precisely, for $\alpha \in [\ln 2, 1]$, define
 \begin{align}
    M_1(n,k,\alpha) &= \frac{1}{-\ln(1-2\ee^{-\alpha}+2\ee^{-2\alpha})} k \ln \frac nk, \notag\\
    M_2(n,k,\alpha) &= \frac{1}{-\ln(1- 2\ee^{-\alpha}+ 2\ee^{-\alpha(1+1/k)})} \ln nk,\label{M123}\\
     M(n,k) &= \min_{\alpha \in [\ln 2, 1]} \max\left\{M_1(n,k,\alpha),M_2(n,k,\alpha)\right\}. \notag 
  \end{align}
Then for any $\epsilon, \delta > 0$, for $n$ sufficiently large, and $m > (1+\delta)M(n, k)$, there exists and $m \times n$ $\epsilon$-almost $k$-separable matrix.
\end{theorem}

Consider the special case $\alpha = \ln 2$. It is possible to see that $M_2$ dominates, and hence that 
there exist almost separable matrices with $m = (1+\delta)
k \log_2 nk$ rows. 
Note that this is sufficient to show the $m = O(k \log n)$ result -- and comes with a slightly easier proof
than the general case (see below). This bound also meets the Malyutov--Seb\H o result of $m \sim k \log_2 n$ for $k$ constant.
However, it is possible to get slightly better constants for most $k = k(n)$ by allowing 
different values of $\alpha$.
In particular, $M_2$ with $\alpha = 1$ gives the best result in many regimes.

In Section \ref{gtsec} we discuss the constants in more detail in the regime $k = n^{1-\beta}$
for $\beta \in (0,1)$. (The reader may wish to skip ahead to Figure \ref{rategraph}, to get a feeling for this result.)

Our proof gives a randomised construction where the matrix is chosen to have entries sampled
from IID Bernoulli random variables; we discuss this in the next section.

\section{Proof of main result}\label{secproof}

We proceed to prove Theorem \ref{mainthm} as follows. Fix $n$ and $k$.
We will choose $\mat A$ to be an $\mbyn$ matrix (where $m$ will be determined
later) with each entry independently $\one$ with probability $p$ and $\zero$ with
probability $q = 1-p$, for some $p$ also to be chosen later. We aim to show that there is a choice of $m$ and $p$ so that, with positive probability, $\mat A$ is $\epsilon$-almost $k$-separable, and hence that such a matrix exists.

The following bound will be important, and is fairly well known -- see for example
Seb\H o \cite{sebo}, who analyses its asymptotics for fixed $k$ as $n \to \infty$.

\begin{lemma}
  Let $\mat A$ be a randomly chosen matrix in $\{0,1\}^{m \times n}$ with each entry independently $1$ with probability $p$.   For any set $\K$ of size $k \leq n/2$,  
  then
    \begin{equation} \label{overlap}
    \Pr(\exists\ \L \text{ with } |\L| = k,\ S(\L) = S(\K)) \leq \sum_{b=0}^{k-1} \binom{k}{b}\binom{n-k}{k-b}
        \left(1 - 2q^k + 2q^{2k-b}\right)^m .
  \end{equation}
\end{lemma}

\begin{proof}
Say that an \emph{overlap} occurs if there exists $\L$ with $|\L| = k$ and $S(\L) = S(\K)$.
  Take two distinct sets $\K, \L$, both of size $k$, that
  have $b = |\K \cap \L|$ elements in common.
  Then a row $j$ of $\mat A$ could distinguish between $\K$ and $\L$ in two ways:
  either we have $j \in S(\K)$ while $j \notin S(\L)$, or the other way round:
  $j \in S(\L)$ while $j \notin S(\K)$.
  
  If the entries of the row $\vec a_j$ are IID Bernoulli$(p)$, these two events each occur
  with probability $q^k(1 - q^{k-b}) = q^k - q^{2k - b}$. Hence, row $j$ fails
  to distinguish between $\K$ and $\L$ with probability $1 - 2q^k(1 - q^{k-b}) = 1 - 2q^k + 2q^{2k-b}$.
  
  Since the rows of $\mat A$ are IID, the whole matrix fails to distinguish between $\K$
  and $\L$ with probability $(1 - 2q^k + 2q^{k-b})^m$.
  
  The result then follows by a union bound over $\L$, noting that the number of
  sets of size $k$ sharing $b$ elements with $\K$ is precisely $\binom kb \binom{n-k}{k-b}$. \qed
\end{proof}

The main work in this paper is a careful asymptotic analysis of the overlap
probability \eqref{overlap}, showing for which $m$ it can be made arbitrarily small.

\begin{lemma} \label{hardlem}
  For every sequence $k = k(n) = o(n)$, $\varepsilon, \delta > 0$, there exists $n_0$ so that if $n > n_0$ and $m > (1+\delta)M(n,k)$, with $M(n,k)$
  as in Theorem \ref{mainthm},
 then $\Pr(\mathrm{overlap}) < \varepsilon$. 
\end{lemma}


\begin{proof}
We first prove that it suffices to have $m > (1+\delta) M_2(n,k,\ln 2)$, with $M_2(n,k,\ln 2) = (1+o(1))k \log_2 nk$. This is simpler to prove than the full result and illustrates the main techniques.

Here, we take $p = 1 - 2^{-1/k}$,
as does Seb\H o \cite{sebo}, so that $q = 2^{-1/k}$. This is a special case of the 
general value of $p$ used in the appendix, $p=1-\ee^{-\alpha/k}$, by taking $\alpha = \ln 2$.
Note that, in group testing parlance, this
is the value of $p$ that gives a $50:50$ chance of a test being positive.
The bound \eqref{overlap} then becomes
  \[ \mathbb P(\text{overlap})
       \leq \sum_{b=0}^{k-1} \binom kb \binom{n-k}{k-b} \left( \frac12 2^{b/k} \right)^m . \]
     
It will be convenient to write $c = k-b$ for the number of nonoverlapping items, to get
  \begin{align*}
    \mathbb P(\text{overlap})
      &\leq \sum_{c=1}^k \binom{k}{k-c} \binom{n-k}{c} \left( \frac12 2^{(k-c)/k} \right)^m \\
      &= \sum_{c=1}^k \binom{k}{c} \binom{n-k}{c} 2^{-cm/k} .
  \end{align*}
  
 When $m > (1+\delta)k \log_2 nk$, then the terms in the above sum are decreasing since
 \begin{align*}
 \frac{\binom{k}{c+1}\binom{n-k}{c+1}2^{-(c+1)m/k}}{\binom{k}{c}\binom{n-k}{c}2^{-cm/k}}
 	&=\frac{(k-c)(n-k-c)2^{-m/k}}{(c+1)^2}\\
    &\leq \frac{c^2 - nc + k(n-k)}{nk(c^2+2c+1)} &&\text{(since $2^{-m/k} \leq 1/nk$)}\\
    &\leq \frac{1}{4},
 \end{align*}
 for $n > 2k$ and $k \geq 2$.
 Thus, the probability of an overlap can be estimated by the largest term with
 \begin{align*}
 \mathbb{P}(\text{overlap})
 	&\leq k(n-k)2^{-m/k} \sum_{c=1}^k  \left(\frac{1}{4}\right)^{c-1}\\
    & \leq k n 2^{-(1+\delta)\log_2 nk} \frac{4}{3}\\
    & =  nk (nk)^{-1-\delta} \frac{4}{3}\\
    & \leq 2 (nk)^{-\delta} ,
	\end{align*}
which, for fixed $\delta > 0$, can be made arbitrarily small for $n$ sufficiently large.    

  


Further, since $\log_2 nk \leq 2 \log_2 n$, we see that 
$m > (1+\delta)k \log_2 nk = O(k \log n)$.


We can get the more general result that it suffices to have $m > (1+\delta)M(n,k)$,
with $M(n,k)$ as in \eqref{M123}, by instead taking $p = 1-\ee^{-\alpha/k}$, and then optimising over $\alpha$. The analysis is very
similar to that above, but somewhat more longwinded. The interested reader is directed to
the appendix for the details. \qed 
\end{proof}

Proving our main result is now straightforward.

\begin{proof}[Proof of Theorem \ref{mainthm}]
Choose the matrix $\mat A$ at random as above, with $m$ and $n$ chosen as in Lemma \ref{hardlem} so that the
overlap probability is at most $\epsilon/2$.

Write $X$ for the number of sets $\K$ of size $k$ that experience an overlap.
It is clear $\mat A$ will be $\epsilon$-almost $k$-separable provided
that $X \leq \epsilon \binom nk$.

Then we have
  \[ \Pr\left( X> \epsilon \binom nk \right)
       \leq \frac{1}{\epsilon \binom nk} \mathbb E X , \]
by the Markov inequality. But this expectation is, by Lemma \ref{hardlem}
  \[ \mathbb E X
      = \sum_{|\K| = k} \Pr(\K \text{ has an overlap}) 
      = \binom nk \Pr(\text{overlap}) 
      \leq \binom nk \frac{\epsilon}{2} .
 \]
Hence, our random $\mat A$ is $\epsilon$-almost $k$-separable with probability at least
$1/2$, so such matrices must exist. \qed
\end{proof}

\section{Rates for nonadaptive group testing} \label{gtsec}

In this section, we show how the use of almost separable matrices can
give new results on the rate of nonadaptive group testing.

As we outlined in the introduction, in a nonadaptive group testing procedure we aim to 
find a subset $\K$ of $k$ defective items within a population of $n$ identical items.
We use $m$ pooled tests. Recall that the outcome of a test $j$ is positive if one or more of the
defective items is in the test pool, and negative if none of them are. 
We summarise our testing procedure by a matrix $\mat A = (a_{ij})$, where
$a_{ij} = \one$ denotes that item $i$ is in the pool for test $j$, and $a_{ij} = \zero$ denotes
that it is not. Recalling the notation of Definition 1, the set of positive tests for a defective set $\K$ is precisely $S(\K)$.

The aim is, given the outcomes $S(\K)$ and the matrix $\mat A$, to identify the defective set $\K$. 
Clearly if there is no other $\L$ with $S(\K) = S(\L)$, then we can find $\K$ (at least theoretically:
for study of practical algorithms for this, see, for example, \cite{ABJ,chan,sejd,malyutov2,wadayama}).
Conversely, if there is an $\L$ with $S(\K) = S(\L)$, then our error probability is at least $1/2$.

A comprehensive survey 
of combinatorial group testing is given in \cite{hwang}. Likewise, the study of nondeterministic 
testing schemes is addressed in the field of probabilistic group testing -- see for example 
\cite{zhigljavsky} and references therein. The derivation of both non-constructive results and 
practical algorithms has been addressed in different contexts, including combinatorial \cite{hwang,malyutov,mazumdar,sebo}, 
probabilistic \cite{ABJ,zhigljavsky} and information-theoretic \cite{atia,BJA,malyutovsurvey,porat,sejd} scenarios.

The connection between separable matrices and nonadaptive group testing
is well explored. In particular, if there are known to be exactly $k$
defective items, then a testing matrix will allow us to find the defective set
with certainty if and only if it is $k$-separable. 
The advantages of using what we call almost separability for group testing in the fixed-$k$ regime have also been discussed in \cite{zhigljavsky}. 

While separable matrices allow detection with zero probability of error, 
the study of group testing within the scope of information theory and the need for efficient 
algorithms generated an interest
in nonadaptive group testing with low -- but not necessarily zero -- probability of error,
a situation which has gained considerable attention \cite{ABJ,chan,sejd,malyutov,malyutov2,atia,wadayama}. Here
the probability of error is defined as an average over all possible defective sets of size $k$; that is,
\[ \Pr(\text{error}) = \frac{1}{\binom nk}\sum_{|\K| = k} \Pr(\text{error} \mid \K)\ .\]

Baldassini, Johnson and Aldridge \cite{BJA} introduced a concept of
the \emph{rate} of group testing to quantify how well a group testing
design works. (An earlier definition of rate for the fixed $k$ regime had been introduced 
by Malyutov \cite{malyutovsurvey}.) The rate is the ratio of the number of tests to the counting bound
$\log_2 \binom nk$. If we interpret the counting bound as a binary labelling of 
all possible defective sets of size $k$, the rate can be considered as the number
of bits learned per test by the group testing procedure.

\begin{definition} \label{defrate}
Consider a group testing problem with $n$ items of which $k$ are defective.
A design with $m$ tests is said to have \emph{rate}
$R=m/\log_2 \binom nk$ .
  
Given a sequence of group testing problems for $n$ items
of which $k = k(n)$ are defective, a rate $R$ is said to be \emph{achievable} for a design $\mat A$ if,
for any $\epsilon > 0$, the design finds the defective set with
error probability at most $\epsilon$ with rate at least $R$ for $n$
sufficiently large.
\end{definition}

We follow Baldassini et al. \cite{BJA,ABJ} and study achievable rates in regimes where
$k = k(n) = n^{1-\beta}$ for different values of the sparsity
parameter $\beta \in (0,1]$.

Note from the above that using a $k$-separable matrix with
$m \geq \Omega (k^2 \log n /$\\$ \log k)$ tests gives rate $0$ for all
values of $\beta < 1$.

As far as we are aware, the best known rate for nonadaptive group testing
until now is achieved by the \texttt{DD} algorithm of Aldridge, Baldassini and Johnson \cite{ABJ},
which has a lower bound on the maximum achievable rate of
  \begin{equation} \label{DD}
    R_{\mathtt{DD}}(\beta) = \frac{1}{\ee \ln 2}
                \min \left\{ \frac{\beta}{1-\beta}, 1 \right\}
            \approx 0.53 \min \left\{ \frac{\beta}{1-\beta}, 1 \right\} ,
  \end{equation}
together with the Malytuov--Seb\H o result that $R = 1$ can be achieved in
the fixed-$k$ regime.
    
Baldassini, Johnson and Aldridge \cite{BJA} also showed that for adaptive
group testing, the generalized binary splitting algorithm of Hwang \cite{hwang}
gives a rate of $1$ (the best possible) for all $\beta \in (0,1]$.

From Theorem \ref{mainthm}, we know that  using an $\epsilon$-almost $k$-separating matrix will find
the defective set with error probability at most $\epsilon$, since the sets
$\K$ without overlaps can by definition be recovered with certainty. Hence, the number of
rows of the almost separating matrix gives bounds on the rate.
Therefore, using our above results, we have the following:

\begin{figure}
\begin{center}
  \includegraphics[scale=.41]{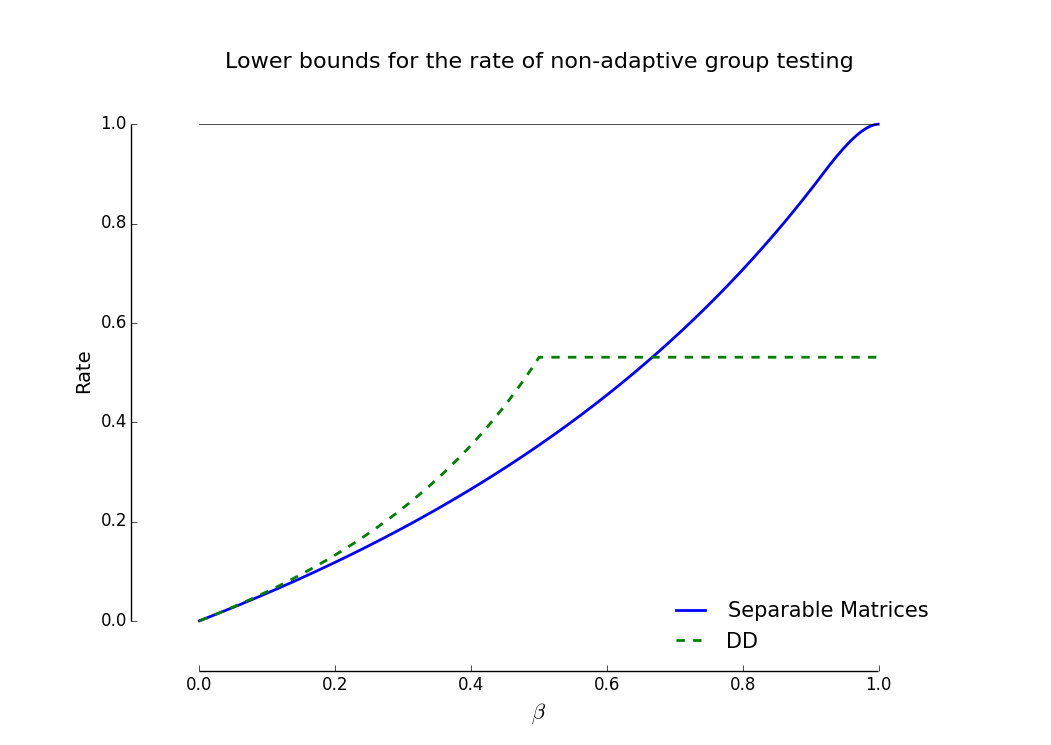}
  \caption{Bounds on rates of group testing, showing the DD bound \eqref{DD} of Baldassini et al, and our new result Theorem \ref{ratethm}.} \label{rategraph}
\end{center}
\end{figure}

\begin{theorem} \label{ratethm}
For $\beta \in (0,1]$ and $k = n^{1-\beta}$, the maximum achievable rate of nonadaptive group testing with $n$ items of which $k$ are defective is bounded below by
  \begin{equation} \label{rates}
     R \geq \frac{1}{\ln 2} \max_{\alpha \in [\ln 2, 1]} \min
         \left\{  2\alpha \ee^{-\alpha} \frac{\beta}{2 - \beta},
            -\ln\left(1 - 2\ee^{-\alpha} + 2\ee^{-2\alpha}\right) \right\} .
  \end{equation}
\end{theorem}

Figure \ref{rategraph} illustrates the result of Theorem \ref{ratethm}.
Note that our result improves over the best known result for $\beta > 2/3$, and meets the Malyutov--Seb\H o\ point as $\beta \to 1$.

\begin{proof}
Following directly from Theorem \ref{mainthm} and the definition of rate, we have
  \begin{multline*}
     R \geq \frac{1}{\ln 2} \max_{\alpha \in [\ln 2, 1]} \min
         \Big\{ -\ln\left(1-2\ee^{-\alpha} + 2 \ee^{-\alpha(1+1/k)}\right) k \frac{\beta}{2 - \beta}, \\
            -\ln\left(1 - 2\ee^{-\alpha} + 2\ee^{-2\alpha}\right) \Big\} ,
  \end{multline*}
noting that, when $k = n^{1- \beta}$,
  \[ k \log_2 nk = \frac{2-\beta}{\beta} k \log_2 \frac nk . 
  \]
  
When $\beta = 1$, the second term is the minimum. When $\beta < 1$, since we have that $k \to \infty$, we can take
limits in the first minimand. We have
  \begin{align*}
    -\ln \big( 1-2\ee^{-\alpha} &+ 2 \ee^{-\alpha(1+1/k)}\big) k \\
      &\quad = -\ln\left( 1-2\ee^{-\alpha} + 2\ee^{-\alpha}\ee^{-\alpha/k} \right) k \\
      &\quad = -\ln\left( 1-2\ee^{-\alpha} + 2\ee^{-\alpha} \left(1 - \frac\alpha k + o\left(\frac1k\right)\right) \right) k \\
      &\quad = -\ln\left( 1 - 2\ee^{-\alpha} \frac\alpha k + o\left(\frac1k\right) \right) k \\
      &\quad = \left( 2\ee^{-\alpha} \frac\alpha k + o\left(\frac1k\right) \right) k \\
      &\quad \to 2\alpha \ee^{-\alpha} .
  \end{align*}
The result follows. \qed
\end{proof}

\begin{figure}
\begin{center}
  \includegraphics[width=0.95\textwidth]{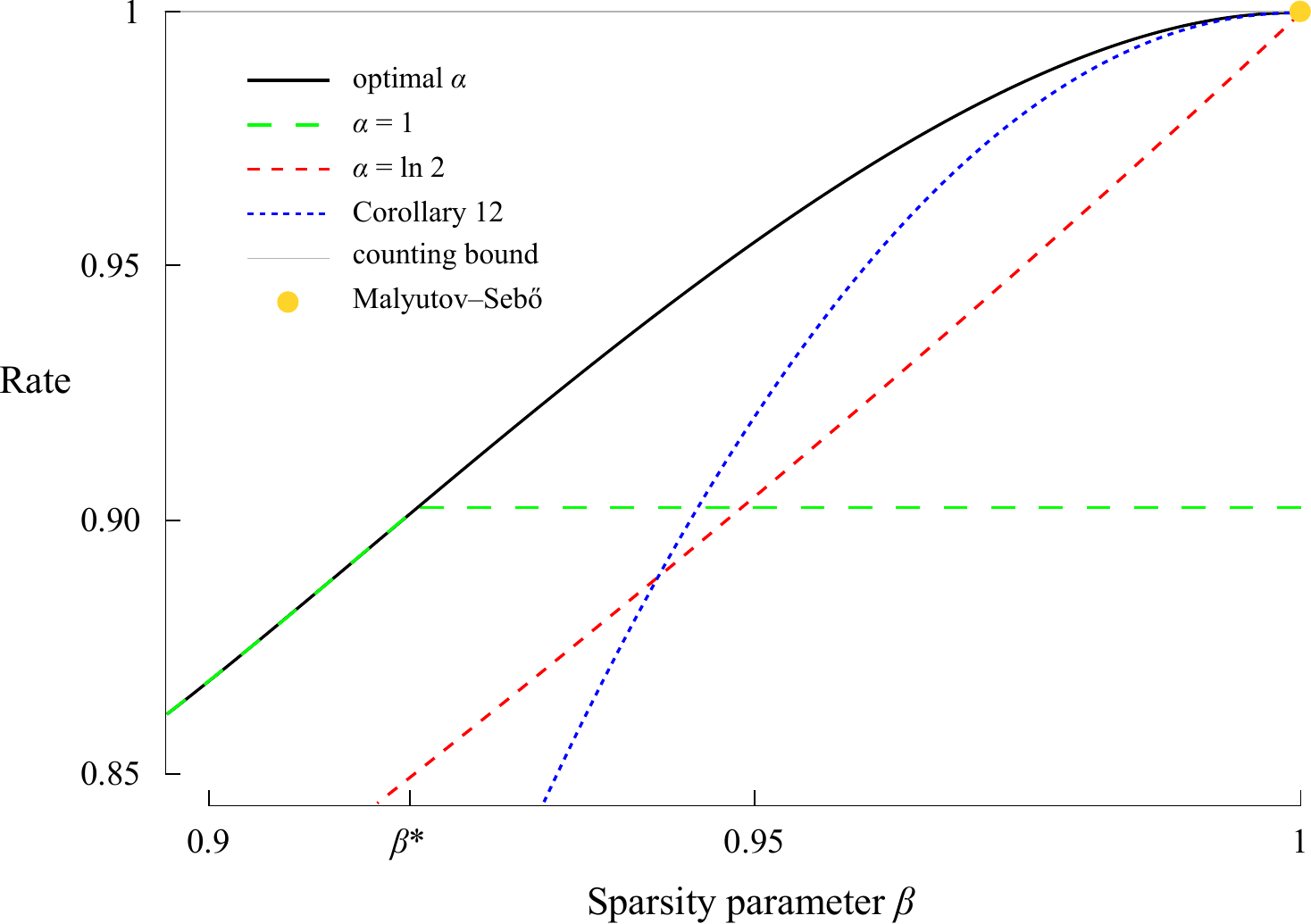}
  \caption{Bounds on rates of group testing for large $\beta$, showing Theorem \ref{ratethm} for different values of $\alpha$ and the approximation of Corollary \ref{cor:beta-1}.} \label{graph2}
\end{center}
\end{figure}

Note that our `simpler' result with $\alpha = \ln 2$ gives a bound
almost as good the general case, namely
  \begin{equation*} \label{a2} R(\ln 2) = \frac{\beta}{2 - \beta} .\end{equation*}
In particular, this choice of $\alpha = \ln 2$ is optimal at $\beta = 1$.

Note also that for all but the sparsest cases, we get the bound by taking 
$\alpha = 1$. Specifically, for $\beta \leq \beta_0$, where
  \[ \beta_0 =\frac{-2 \ln(1- 2\ee^{-1} + 2\ee^{-2})}{2\ee^{-1}-\ln(1- 2\ee^{-1} + 2\ee^{-2})} \approx 0.92, \]
the best value of the bound is
  \begin{align*} R(1) &= \frac{1}{\ln 2} \min
                \left\{ 2\ee^{-1} \frac{\beta}{2 - \beta},
                 -\ln\left(1 - 2\ee^{-1} + 2\ee^{-2}\right) \right\} \label{a1} \\
          &= \frac{1}{\ln 2}  2\ee^{-1} \frac{\beta}{2 - \beta} \\
          &\approx 1.06     \frac{\beta}{2 - \beta}. \end{align*}
       
For $\beta \in (\beta_0, 1)$, the optimal rate is given as the maximum in \eqref{rates}, 
and the optimal $\alpha$ is that which achieves the maximum. 
It's easy to see for $\beta \geq \beta_0$ that the maximum over $\alpha$ is achieved when the two terms in the minimum
are equal, and this is simple to solve numerically. 
However, here we also provide some closed form approximations to this which could be useful. 

\begin{corollary}\label{cor:beta-1}
For $\frac{2\ln 2}{1+\ln 2}< \beta < 1$ and $k = n^{1-\beta}$, the maximum achievable rate of nonadaptive group testing with $n$ items, of which $k$ are defective, is bounded from below by
\[
R \geq 1 - \frac{1}{\ln 2}\ln\left(1+ \left(\frac{2(1-\beta) \ln 2}{\beta(1-\ln 2)} \right)^2 \right)
\]
\end{corollary}

This is illustrated in Figure \ref{graph2}. From this, we see that the bound of Corollary \ref{cor:beta-1} is very good for $\beta \approx 1$, but that when $\beta$ is not much above $\beta_0$, then the bound of simply $\alpha = 1$ is better.
Hence, setting
  \[ \beta_1 = \frac{2 \ln 2}{1 - 2\ee^{-1} + \ln 2 + 2\ee^{-1} \ln 2} \approx 0.94, \]
and taking $\beta_0$ as above, we get the following bound:


\begin{corollary}\label{cor:beta-cases}
For $\beta \in (0,1)$ and $k = n^{1-\beta}$, the maximum achievable rate of nonadaptive group testing with $n$ items, of which $k$ are defective, is bounded from below by
\[
R \geq 
\frac{1}{\ln 2}
\begin{cases}
\displaystyle 2\ee^{-1}\frac{\beta}{2-\beta}	&\text{if } \beta \leq \beta_0 \\
-\ln(1 - 2\ee^{-1} + 2\ee^{-2})	&\text{if } \beta_0 < \beta \leq \beta_1, \\
\displaystyle \ln 2 - \ln\left(1+ \left(\frac{2(1-\beta) \ln 2}{\beta(1-\ln 2)} \right)^2 \right)	&\text{if } \beta > \beta_1
\end{cases}
\]
\end{corollary}

The proofs of these
statements can be found in Appendix B.


\section{Conclusions and further work}
We have explored the asymptotics of almost separability and we have shown that 
almost separable matrices exist with $O(k \log n)$ rows. Furthermore, we have proved that the use 
of almost separable matrice can improve the lower bounds on the rate of nonadaptive
group testing in the very sparse regime.

Several interesting questions, however, remain still open, and provide scope for 
future research. Most notably, while we have given new achievable rates, the maximum rate
of nonadative group testing is still unknown. In particular, we know of no upper bounds
beyond the trivial counting bound.

As discussed in Section 2, Chen and Hwang \cite{chen} have proved that 
disjunct and separable matrices share the same asymptotics by showing how to 
construct a $k$-disjunct matrix out of a $2k$-separable matrices by adding at most one 
row to it. Unlike its inverse (disjunctness implying separability),
this statement doesn't naturally carry through to the case of almost 
separability/disjunctness.

Another problem is to extend the existing results to other
regimes than the $k = n^{1-\beta}$ for $\beta \in (0,1]$ considered here.
Of particular interest is the case where $k = cn$ grows like a constant proportion of $n$,
as in recent work by Wadayama \cite{wadayama}. Note that the counting bound
now gives a lower bound of $m = O(n)$, while, for coupon-collector reasons,
the IID random approach here inevitably leads to the suboptimal $m = \Omega(n \log n)$.

\appendix



\section{Asymptotic analysis of the overlap probability}


We now show the full result of Lemma \ref{hardlem}.

We use the same random construction as the special case described in Section \ref{secproof}, 
but now take $p = 1 - \ee^{-\alpha/k}$,
so $q = \ee^{-\alpha/k}$, where
$\alpha$ is a parameter to be chosen later (simply taking $\alpha = \ln 2$ as in Section 
\ref{secproof} gives
$p = 1 -2^{-1/k}$). Within the group testing literature, different values of $p$ 
have also been considered. For example, the value $p = 1/k$ 
(which gives an average of one defective per test) has been considered before by many authors 
\cite{ABJ,chan,atia,zhigljavsky}, while Sejdinovic and Johnson \cite{sejd} 
consider the more general
$\alpha/k$ for noisy group testing. The same value can be obtained asymptotically 
in this context, as $p \sim \alpha/k$ if $k \to \infty$ as $n \to \infty$. 

\begin{proof}[Proof of Lemma \ref{hardlem}]
We wish to find values of $m$ such that $\Pr(\text{overlap})$ can be made
arbitrarily small. It will be convenient to write 
\[
s = 1-2q^k = 1-2\ee^{-\alpha}, \qquad 
t = 2q^{2k} = 2\ee^{-2\alpha}, \qquad 
u = \frac{1}{q}=\ee^{\alpha/k} ,
\]
allowing us to rewrite the bound \eqref{overlap} as
    \begin{equation*} 
    \Pr(\text{overlap}) \leq \sum_{b=0}^{k-1} \binom{k}{b}\binom{n-k}{k-b}
        \left(s +tu^b\right)^m .
  \end{equation*}
%
%
%
%
%
As before, it will be more convenient to deal with $c = b-k$, which gives
    \begin{equation} \label{overlap2} 
    \Pr(\text{overlap}) \leq \sum_{c=1}^{k} \binom{k}{c}\binom{n-k}{c}
        \left(s +tu^{k-c}\right)^m .
  \end{equation}
  
Now, we expand out $(s+tu^b)^m$ in \eqref{overlap2} using the binomial
theorem and reverse the order of summation to get
\begin{align}
  \Pr(\text{overlap})
  &\leq \sum_{c=1}^{k} \binom{k}{c} \binom{n-k}{c} \sum_{j=0}^m \binom{m}{j} s^{m-j} t^j u^{(k-c)j} \notag\\
  &= \sum_{j=0}^m \binom{m}{j} s^{m-j} t^j \sum_{c=1}^{k} \binom{k}{c} \binom{n-k}{c} u^{(k-c)j} \notag\\
  &= \sum_{j=0}^m \binom{m}{j} s^{m-j} t^j u^{jk} \sum_{c=1}^{k} \binom{k}{c} \binom{n-k}{c} q^{cj} 
  \label{eq:main}
\end{align}

Consider the inner sum of \eqref{eq:main}.
It is possible to approximate it by its largest term, which will depend
on the value of $j$.  To start with, the following bound holds:
\begin{equation} \label{above}
 \binom{k}{c} \binom{n-k}{c} q^{cj} \leq \left(\frac{\ee^2 k n q^j}{c^2} \right)^c.
\end{equation}
Note that for any $a$, the function $(a/x^2)^x$ attains its maximum at $x = \sqrt{a}/\ee$; 
and further is increasing for $x < \sqrt{a}/\ee$ and decreasing for $x> \sqrt{a}/\ee$.  
In \eqref{above}, the maximum corresponds to $c = \sqrt{k n q^j}$.  
Now, $1 < \sqrt{k n q^j} < k$ when
\[
  \frac{1}{-\ln q} \ln \frac nk < j
    < -\frac{1}{-\ln q} \ln nk ,
\]
or, since $q = \ee^{-\alpha/k}$,
\[
  \frac{1}{\alpha}k \ln \frac nk < j
    < \frac{1}{\alpha} k \ln nk .
\]
Then, in light of the above,
we will split between the three cases: 
first, $j \leq k/\alpha \ln n/k$; second, $k/\alpha \ln n/k < j < k/\alpha \ln nk$; 
and third, $j \geq k/\alpha \ln nk$.

For the first case, $j \leq k/\alpha \ln n/k$, the maximum of \eqref{above} is attained at $c=k$, giving the bound
\[
 \left( \frac{\ee^2 k n q^j}{k^2}\right)^{k} = \ee^{2k} \left(\frac nk \right)^{k} q^{jk} .
\]
Summing over this range for $j$ yields
\begin{align*}
 \sum_{j=0}^{k/\alpha \ln n/k} \binom{m}{j} s^{m-j} t^j u^{jk} k\ee^{2k} \left(\frac nk \right)^{k} q^{jk} 
  &= k\ee^{2k} \left(\frac nk \right)^{k}
          \sum_{j=0}^{k/\alpha \ln n/k} \binom{m}{j}s^{m-j} t^j \\
  &\leq k\ee^{2k} \left(\frac nk \right)^{k}
          \sum_{j=0}^{m} \binom{m}{j}s^{m-j} t^j \\
  &= k\ee^{2k} \left(\frac nk \right)^{k} (s+t)^m \\
  &= k\exp\left(2k + k\ln \frac nk + m\log(s+t) \right).
\end{align*}
Provided that
  \begin{align}
    m &> (1+\delta)\frac{1}{-\ln(s+t)} k \ln \frac nk \notag \\
      &= (1+\delta)\frac{1}{-\ln(1-2\ee^{-\alpha} + 2\ee^{-2\alpha})} k \ln \frac nk \notag \\
      &= (1+\delta) M_1(n,k,\alpha), \label{cond2}
  \end{align}
for some $\delta > 0$, then this can be made arbitrarily small for $n$ sufficiently large.

For the second case, $k/\alpha \ln n/k < j < k/\alpha \ln nk$,
the maximum is attained at $c = \sqrt{k n q^j}$, giving the bound
\begin{multline*}
\left(\frac{\ee^2 k n q^j}{k n q^j} \right)^{\sqrt{k n q^j}} = \exp(2 \sqrt{k n q^j}) 
   \leq \exp(2 \sqrt{k n q^{k/\alpha \ln n/k}}) \\ 
   = \exp \left( 2 \sqrt{k n \left(\frac nk\right)^{k/\alpha \ln q}} \right)
   = \exp\left(2 \sqrt{kn \frac kn}\right) = \exp(2k).
\end{multline*}
Then we have that
\begin{align*}
\sum_{j = k/\alpha \ln n/k}^{k/\alpha\ln nk} \binom{m}{j} s^{m-j} t^j u^{jk}  
\sum_{c=1}^{k} &\binom{k}{c} \binom{n-k}{c} q^{jc}\\
	&\leq \sum_{j = k/\alpha \ln n/k}^{k/\alpha\ln nk} \binom{m}{j} s^{m-j} (t u^k)^j k\ee^{2k}
       \\
	&= k\ee^{2k} \mathbb{P}\left(\frac{1}{\alpha}k \ln \frac nk < X \leq \frac{1}{\alpha} k \ln nk\right) ,
\end{align*}
where we called $X \sim \operatorname{Bin}(m, tu^k)$, and we have used that $s = 1 - tu^k$. 
%
%
Then as long as 
  \begin{equation}  \label{expect}
  \mathbb EX = mtu^k > (1+\delta)\frac{1}{\alpha} k \ln nk,
  \end{equation}
we have by the Azuma--Hoeffding inequality that
  \begin{align*}
     k\ee^{2k} \mathbb{P}\left(\frac 1\alpha k \ln \frac nk < X \leq \frac1\alpha k \ln nk \right)
       &\leq k\ee^{2k} \mathbb{P}\left(X \leq \frac1\alpha \ln nk  \right)\\
       &\leq k\ee^{2k} \exp \left( - \frac 2m \left(mtu^k -  \frac1\alpha k\ln nk \right)^2 \right) \\
       &= k\exp \left(2k - 2m \left(tu^k- \frac{k \ln nk}{\alpha m} \right)^2 \right) .
  \end{align*}
Given \eqref{expect}, this can be made arbitrarily small for $n$ sufficiently large. We can rewrite \eqref{expect} as
  \begin{equation} \label{cond3}
  m > (1+\delta) \frac{1}{\alpha tu^k} k \ln nk
    = (1+\delta) \frac{\ee^\alpha}{2\alpha} k \ln nk = (1+\delta) M_2(n,k,\alpha) .
  \end{equation}

Now for the final case, when $j \geq k/\alpha \ln nk$. 
Note that for $j \geq k/\alpha \ln nk$, 
\[
q^j \leq q^{k/\alpha\ln nk} = \ee^{- \ln nk} = \frac{1}{nk},
\]
hence $nkq^j\leq1$. 
Then, splitting up $c =1$, $c = 2$, and $c \geq 3$, and noting that $\ee^2/9 < 1$, we have
\begin{align*}
\sum_{c=1}^{k} \left(\frac{\ee^2 kn q^j}{c^2} \right)^c
    &\leq \ee^2 kn q^j \left(1 + \frac{\ee^2 kn q^j}{2^4} + \sum_{c=3}^{k} \frac{1}{c^2} \left(\frac{\ee^2 kn q^j}{c^2}\right)^{c-1} \right)\\
    &\leq \ee^2 kn q^j \left( 1 + \frac{\ee^2}{16} + \frac{1}{9}\sum_{c=3}^{\infty}\left(\frac{\ee^2}{9} \right)^{c-1}\right)\\
    &\leq 5 \ee^2 kn q^j.
\end{align*}
Thus,
\begin{align*}
\sum_{j=\alpha\ln nk}^m \binom{m}{j} s^{m-j} t^j u^{jk} \sum_{c=1}^{k} \left(\frac{\ee^2 k n q^j}{c^2} \right)^c
	&\leq \sum_{j=\alpha\ln nk}^m \binom{m}{j} s^{m-j}t^j u^{jk} 5 \ee^2 kn q^j \\
    &\leq 5 \ee^2 kn \sum_{j=0}^m \binom{m}{j} s^{m-j} (tu^{k-1})^j \\
    &= 5 \ee^2 kn (s+tu^{k-1})^m \\
    &= 5 \exp \left( \ln nk + m \ln (s+tu^{k-1})\right)
\end{align*}
To make this small requires
  \begin{equation} \label{cond4a}
    m > (1+\delta) \frac{1}{-\ln(s + tu^{k-1})} \ln nk .
  \end{equation}

In order to compare the condition in \eqref{cond4a} to \eqref{cond2} and \eqref{cond3}, note that for any $x, y \in (0,1)$,
\[
 - \ln(1-x(1-\ee^{-y})) \leq xy.
\]
The above inequality can be seen, for example, since for each $y$, the function $f_y(x) = xy+\ln(1-x(1-\ee^{-y}))$ is concave for $x \in [0,1]$ with $f_y(0) = 0 = f_y(1)$.  Thus, since $s+tu^{k-1} = 1 - 2\ee^{-\alpha}(1-\ee^{-\alpha/k})$, then
\[
 \frac{-1}{\ln(s+tu^{k-1})} = \frac{-1}{\ln(1-2\ee^{-\alpha}(1-\ee^{-\alpha/k}))} \geq \frac{k}{2 \ee^{-\alpha} \alpha}.
\]
Thus, condition \eqref{cond4a} is always stronger than \eqref{cond2} and one can see that when $k$ tends to infinity, the two conditions are asymptotically equal.
 
Hence from \eqref{cond2}, \eqref{cond3}, and \eqref{cond4a} 
our requirements are
  \[
    m > (1+\delta) M_1(n,k,\alpha) \qquad
    m > (1+\delta) M_2(n,k,\alpha) .
  \]
From the above, we can optimise this result over $\alpha$. Noting that $M_1$ is minimised at $\alpha = \ln 2$ and $M_2$ is minimised at $\alpha = 1$, it is sufficient to just consider $\alpha \in [\ln 2, 1]$.

This proves Lemma \ref{hardlem}. \qed
\end{proof}


%
%


\section{Explicit bounds on rate}

Here we give the proofs of Corollaries \ref{cor:beta-1} and \ref{cor:beta-cases}.

\begin{proof}[Proof of Corollary \ref{cor:beta-1}]
The bound on $R$ follows from Theorem \ref{ratethm} by a careful choice of $\alpha$ in terms of $\beta$.

In order to simplify some of the expressions that follow, define $y = y(\alpha) = 1-2\ee^{-\alpha}$ and $t = 1 - \frac{\beta}{2-\beta}$.  Then, for $\alpha \in [\ln 2, 1]$ we have $y \in [0,1 - 2/\ee]$ and as $\beta$ tends to $1$, $t$ tends to $0$.  Further, the expressions in Theorem \ref{ratethm} can be simplified as
\[
 - \ln(1-2\ee^{-\alpha} + 2\ee^{-2\alpha}) = -\ln\left(\frac{1}{2}(1+y^2)\right) = \ln 2 - \ln(1+y^2)
 \]
 and
\begin{equation*}
 2\alpha \ee^{-\alpha} \frac{\beta}{2 - \beta}  = (1-y)\left(-\ln\left(\frac{(1-y)}{2} \right) \right)(1-t)
 		 = (1-y)\left(\ln 2 - \ln(1-y) \right)(1-t).
\end{equation*}

Thus, the result of Theorem \ref{ratethm} can be restated as
\begin{equation}\label{eq:restated-R}
R \geq \frac{1}{\ln 2} \min_{y \in [0,1-2/\ee]}\left\{\ln 2 - \ln(1+y^2), (1-t)(1-y)\left(\ln 2 - \ln(1-y)\right) \right\}
\end{equation}

The desired result then follows from equation \eqref{eq:restated-R} by choosing
\begin{equation}\label{eq:y-def}
y = \frac{\ln 2}{1- \ln 2} \cdot \frac{t}{1-t}.
\end{equation}
Note that, by the definition of $t$, $\frac{t}{1-t} = \frac{2(1-\beta)}{\beta}$.

What remains is to show that for $y$ given by equation \eqref{eq:y-def},
\begin{equation}\label{eq:y-bound-claim}
\ln 2 - \ln(1+y^2) \leq (1-y)(1-t)\left(\ln 2 - \ln(1-y) \right).
\end{equation}

For $y$ given by equation \eqref{eq:y-def}, the right-hand side of equation \eqref{eq:y-bound-claim} is
\begin{align*}
(1-y)&(1-t)\left(\ln 2 - \ln(1-y) \right)\\
    & = (1-y)(1-t)(\ln 2 + y) - (1-y)(1-t)(y+\ln(1-y))\\
    & = (1-t)\ln 2 + y(1-t)(1-\ln2)  - (1-t)y^2\\
    & \qquad - (1-y)(1-t)(y+\ln(1-y))\\
    & = (1-t)\ln 2 + t \ln 2 - (1-t)y^2\\
    & \qquad - (1-y)(1-t)(y+\ln(1-y)) &&\text{(by eq. \eqref{eq:y-def})}\\
    & = \ln 2 - (1-t)(y^2 + (1-y)y + (1-y)\ln(1-y))\\
    & = \ln 2 - (1-t)(y+ (1-y)\ln(1-y))\\
    & = \ln 2 - \left(\frac{\ln 2}{\ln 2 + y(1-\ln 2)}\right)(y+ (1-y)\ln(1-y)) &&\text{(by eq. \eqref{eq:y-def})}.
\end{align*}

Thus, in order to show that the inequality in \eqref{eq:y-bound-claim} holds, it suffices to show that for all $y \in [0,1]$,
\begin{equation}\label{eq:y-bound-claim2}
y+(1-y)\ln(1-y) \leq \left(1+ \frac{y(1-\ln2)}{\ln 2} \right)\ln(1+y^2)
\end{equation}

The inequality in \eqref{eq:y-bound-claim2} is shown by considering separately the cases $y \leq 1/2$ and $y > 1/2$.

Consider first the case $y \leq 1/2$.  Using the fact that $\ln(1-y) < -y$ and 
\[
\ln(1+y^2) \geq y^2 - y^4/2 \geq y^2 - y^3/4 = y^2(1-y/4).
\]
Then,
\[
y + (1-y)\ln(1-y) < y^2
\]
and for all $y \in [0,1/2]$,
\[
1 \leq \left(1+y\frac{(1-\ln 2)}{\ln 2} \right)\left(1 - \frac{y}{4}\right).
\]
Thus, for $y \leq 1/2$,
\begin{multline*}
y+(1-y)\ln(1-y) \leq y^2 \leq y^2  \left(1+y\frac{(1-\ln 2)}{\ln 2} \right)\left(1 - \frac{y}{4}\right)\\ \leq  \left(1+y\frac{(1-\ln 2)}{\ln 2} \right)\ln(1+y^2).
\end{multline*}

Consider now the inequality from \eqref{eq:y-bound-claim2} in the case $y \geq 1/2$.  Note that for all $y \in [0,1]$,
\[
\ln(1+y^2) \geq \ln 2 - (1-y).
\]
The above inequality can be seen to be true since it holds for $y = 0$ and $y = 1$ and $\ln(1+y^2)$ is concave.  Thus, in order to prove the inequality in \eqref{eq:y-bound-claim2}, it suffices to show that for $y \in [1/2, 1]$,
\begin{equation}\label{eq:big-y}
y+(1-y)\ln(1-y) \leq\left(1+y\frac{(1-\ln2)}{\ln 2} \right) (\ln 2 - (1-y)).
\end{equation}
Again, the inequality in equation \eqref{eq:big-y} can be seen to be true since it holds for $y = 1/2$ and $y = 1$ and the function 
\begin{align*}
\bigg(1+ &y\frac{(1-\ln2)}{\ln 2} \bigg)(\ln 2 - (1-y)) - y - (1-y)\ln(1-y)\\
	& = \ln 2 + y(1-\ln 2) - 1 + y - y \frac{(1-\ln 2)}{\ln 2} + y^2 \frac{(1-\ln2)}{\ln 2}\\
    & \qquad - y - (1-y)\ln (1-y)\\
    & = (\ln 2 - 1) + y(1-\ln 2) \left(1 - \frac{1}{\ln 2} \right) + y^2 \frac{(1-\ln 2)}{\ln 2} - (1-y)\ln(1-y)
\end{align*}
is concave for $y \in [0,1]$. \qed
\end{proof}

Next, is the proof of Corollary \ref{cor:beta-cases}.

\begin{proof}[Proof of Corollary \ref{cor:beta-cases}]
For $\beta < \beta_1$, the result follows from Theorem \ref{ratethm} by subsituting $\alpha = 1$ and noting that the inequality
\[
\frac{2\beta}{\ee(2-\beta)} \leq -\ln(1-2/\ee + 2/\ee^2)
\]
holds exactly when $\beta < \beta_0$.

For $\beta \geq \beta_1$, the result follows from Corollary \ref{cor:beta-1} by noting that $\beta_1 > \frac{2\ln 2}{1+\ln 2}$. \qed
\end{proof}

In Corollaries \ref{cor:beta-1} and \ref{cor:beta-cases}, a better bound for the case $\beta > \beta_0$ can be obtained by substituting in Theorem \ref{ratethm}, $\alpha$ chosen so that
\[
1 - 2\ee^{-\alpha} = \frac{-\beta(1-\ln 2) + \sqrt{\beta^2(1-\ln 2)^2 + 4(1-\beta)(4-3\beta)\ln 2}}{4-3\beta},
\]
but the expression obtained does not seem simpler than statement of Theorem \ref{ratethm} itself.


\begin{thebibliography}{29}

\bibitem{ABJ}
  M Aldridge, L Baldassini, and O Johnson.
  Group testing algorithms: bounds and simulations.
  \emph{IEEE Transactions on Information Theory}, \textbf{60}:6, 3671--3687, 2014.

\bibitem{atia}
  GK Atia and V Saligrama.
  Boolean compressed sensing and noisy group testing.
  \emph{IEEE Transactions on Information Theory}, \textbf{58}:3, 1880--1901, 2012.

\bibitem{BJA}
  L Baldassini, O Johnson, and M Aldridge.
  The capacity of adaptive group testing.
  \emph{2013 IEEE International Symposium on Information Theory Proceedings}, 2676--2680, 2013.
  
\bibitem{chan}
  CL Chan, S Jaggi, V Saligrama, and S Agnihotri.
  Non-adaptive group testing: explicit bounds and novel algorithms.
  \emph{IEEE Transactions on Information Theory}, \textbf{60}:5, 3019--3035, 2014
  
\bibitem{chen}
  H-B Chen and FK Hwang.
  Exploring the missing link among $d$-separable, $\overline d$-separable and $d$-disjunct matrices.
  \emph{Discrete Applied Mathematics}, \textbf{155}:5, 662–-664, 2007.

\bibitem{dorfman}
  R Dorfman.
  The detection of defective members of large populations.
  \emph{The Annals of Mathematical Statistics}, \textbf{14}:4, 436--440, 1943.

\bibitem{hwang}
  D-Z Du and FK Hwang.
  \emph{Combinatorial Group Testing and Applications}, second edition.
  Series on Applied Mathematics, \textbf{12}, World Scientific, 2000.
  
\bibitem{dyachkov}
  AG D'yachkov and VV Rykov.
  Bounds on the length of disjunctive codes. 
  \emph{Problems of Information Transmission}, \textbf{18}:3, 166–-171, 1982.
  
\bibitem{erdosmoser}
  P Erd\H{o}s and L Moser.
  Problem 35.
  \emph{Proceedings on the Conference of Combinatorial Structures and their Applications}, Gordon and Breach, 1970.

\bibitem{furedi}
  Z Füredi.
  On $r$-cover-free families. 
  \emph{Journal of Combinatorial Theory, Series A}, \textbf{73}:1, 172--173, 1996.

\bibitem{ks}
  WH Kautz and RC Singleton.
  Nonrandom binary superimposed codes.
  \emph{IEEE Transaction on Information Theory}, \textbf{10}:4, 363--377, 1964.

\bibitem{macula}
  A Macula, V Rykov and S Yekhanin.
  Trivial two-stage group testing for complexes using almost disjunct matrices.
  \emph{Discrete Applied Mathematics}, \textbf{137}:1, 97--107, 2004.

\bibitem{malyutov2}
   D Malioutov and M Malyutov.
   Boolean compressed sensing: Lp relaxation for group testing.
   \emph{2012 IEEE International Conference on Acoustics, Speech and Signal Processing (ICASSP)},
      3305–-3308, 2012.

\bibitem{malyutov}
    MB Malyutov.
  The separating property of random matrices.
  \emph{Mathematical Notes of the Academy of Sciences of the USSR}, \textbf{23}:1, 84--91, 1978.

\bibitem{malyutovsurvey}
  M Malyutov.
  Search for sparse active inputs: a review.
  In H Aydinian, F Cicalese, and C Deppe (Eds),
    \emph{Information Theory, Combinatorics and Search Theory}
    Lecture notes in Computer Science, \textbf{7777}, Springer, 609--647, 2013.

\bibitem{mazumdar}
  A Mazumdar. 
  On almost disjunct matrices for group testing.
  \emph{Algorithms and Computation}, Lecture Notes in Computer Science, \textbf{7676}, 649--658, 2012.
  
\bibitem{porat}
  E Porat and A Rothschild. 
  Explicit Non-Adaptive Combinatorial Group Testing Schemes.
  In L Aceto, I Damgard, LA Goldberg, MM Halldorsson, A Ingolfsdottir and I Walukiewicz (Eds),
  \emph{ICALP 2008}, Lecture Notes in Computer Science, \textbf{5125}, 748--759, 2008.

\bibitem{ruszinko}
  M Ruszink\'o.
  On the upper bound of the size of $r$-cover-free families.
  \emph{Journal of Combinatorial Theory, Series A}, \textbf{66}:2, 302--310, 1994.

\bibitem{sebo}
  A Seb\H o.
  On two random search problems.
  \emph{Journal of Statistical Planning and Inference}, \textbf{11}:1, 23--31, 1985.
  
\bibitem{sejd}
  D Sejdinovic and OT Johnson.
  Note on noisy group testing: asymptotic bounds and belief propagation reconstruction.
  \emph{Proceedings of the 48th Annual Allerton Conference on Communication, Control and Computing},
    998--1003, 2010.

\bibitem{wadayama}
  T Wadayama.
  An analysis on non-adaptive group testing based on sparse pooling graphs.
  \emph{2013 IEEE International Symposium on Information Theory},2681–-2685, 2013.

\bibitem{zhigljavsky}
  A Zhigljavsky.
  Probabilistic existence theorems in group testing.
  \emph{Journal of Statistical Planning and Inference},\textbf{115}:1, 1--43, 2003.

\end{thebibliography}
\end{document}